\documentclass{amsart}

\usepackage{verbatim}
\usepackage{amssymb}
\usepackage{amscd}
\usepackage{amsmath}
\usepackage[alphabetic,initials]{amsrefs}
\newtheorem{theorem}{Theorem}[section]

\newtheorem{cor}[theorem]{Corollary}
\newtheorem{lemma}[theorem]{Lemma}
\newtheorem{proposition}[theorem]{Proposition}
\numberwithin{equation}{subsection}

\title{
On the representation of the number of integral points of an elliptic curve\\modulo a prime number}
\author{Michael Th. Rassias}
\date{\today}
\address{Department of Mathematics, ETH-Z\"{u}rich, R\"{a}mistrasse 101, 8092 Z\"{u}rich, Switzerland.}
\email{michail.rassias@math.ethz.ch}
\thanks{}

\begin{document}

 \maketitle

%
%
%
\section{Introduction}
\vspace{5mm}
\noindent The problem of counting the number of integral points of an elliptic curve over a finite field has captured the interest of both pure number theorists and cryptographers. It is important in cryptography to know the number of integral points in order to evaluate the difficulty of solving the discrete logarithm problem in the group of points on an elliptic curve.\\
The number of applications of the theory of elliptic curves to cryptography over the last decades is overwhelming. Among the most important applications are the ones in public key cryptography (see [9], [10], [14], [21]), primality testing (see [2], [4]) and the factoring of large integers (see [13], [16]).\\
\noindent In this paper we shall investigate the problem of the representation of the number of integral points of an elliptic curve modulo a prime number p. Our first thought when dealing with this problem would be to consider an elliptic curve of the form $y^2=x^3+ax+b\;(\bmod \;p)$, with $a$, $b\in\mathbb{Z}$ and $p$ a prime number, and try to identify for which values $x=0,\:1,\:\ldots,\:p-1$ the integer $x^3+ax+b$ is a square $\bmod\;p$. However, among the nonzero integers $\bmod\;p$ half of them are squares. From the other half, we obtain two roots, namely $y$ and $-y$. Therefore, we expect that the equation $y^2=x^3+ax+b\;(\bmod \;p)$ will be satisfied for approximately $p$ pairs $(x,y)$. In other words, we expect the elliptic curve to have more or less $p+1$ points, counting also the point at infinity.\\
\noindent In 1930, H. Hasse (Math. Z. $\textbf{31}(1930)$, $565-582$) proved the following very deep result (see [19] for a proof).
\begin{theorem}
Let $p$ be a prime number and $E$ an elliptic curve over the finite field $\mathbb{F}_p$. Then
$$\left|\#E(\mathbb{F}_p) -p-1\right|<2\sqrt{p}\:.$$
\end{theorem}
\noindent Moreover, a very interesting result is that whenever there exists an integer $N$, such that 
$$\left|N-p-1\right|<2\sqrt{p}\:,$$
then there exists an elliptic curve $\bmod\;p$ which has exactly $N$ integral points (see also [21]).\\
When the prime number $p$ is considerably small, it is easy to calculate the number of integral points on an elliptic curve $\bmod\;p$, but for large prime numbers this becomes very difficult. An efficient algorithm for calculating the number of integral points over finite fields was discovered in 1985 by R. Schoof (see [18]) and was later improved by N. D. Elkies and A. O. L. Atkin (see[3]).\\
There exist several formulae for the representation of the number of integral points of an elliptic curve $\bmod\;p$. One of the most commonly used is the following.\\
\noindent The number of solutions of the congruence
$$y^2=x^3+ax+b\;(\bmod \;p)\:$$
where $x,\:y\in\mathbb{M}_p=\left\{0,\:1,\:\ldots,\:p-1\right\}$ and $a,\:b\in\mathbb{Z}$ can be expressed in the form (see [15])
\[
N_p=\frac{1}{p}\sum_{x,y=0}^{p-1}\sum_{m=0}^{p-1}e^{2\pi imF(x,y)/p}=p+\frac{1}{p}\sum_{x,y=0}^{p-1}\sum_{m=1}^{p-1}e^{2\pi imF(x,y)/p}\:,
\]
where $F(x,y)=y^2-x^3-ax-b$.\\
Of course, if we are talking about elliptic curves and not just congruences, then since the point at infinity is the neutral element of the additive group of an elliptic curve, we have to add one to the previous  formula obtaining 
\[
N_p=1+\frac{1}{p}\sum_{x,y=0}^{p-1}\sum_{m=0}^{p-1}e^{2\pi imF(x,y)/p}\:.\tag{1}
\]
\noindent It is a known fact (see [6], [11], [12]) that
\begin{eqnarray}
\sum_{y=0}^{p-1}e^{2\pi imy^2/p}&=&\sum_{y=0}^{p-1}\left(\frac{y}{p}\right)e^{2\pi imy/p}\nonumber\\
&=&\left(\frac{m}{p}\right)\sum_{y=0}^{p-1}\left(\frac{y}{p}\right)e^{2\pi iy/p}\nonumber\\
&=&\left\{ 
  \begin{array}{l l}
    \left(\frac{m}{p}\right)\sqrt{p}\:, & \quad \text{if $p\equiv1\;(\bmod\;4)$}\vspace{2mm}\\ 
    \left(\frac{m}{p}\right)i\sqrt{p}\:, & \quad \text{if $p\equiv3\;(\bmod\;4)$}\:,\\
  \end{array} \right.
\nonumber
\end{eqnarray}
where $\left(\frac{\dot{}}{p}\right)$ denotes the Legendre symbol.\\
There is no explicit analytical formula in order to calculate exponential sums which involve polynomials of third degree. Hence, we shall try to find a way to express the exponential sum using non-exponential terms.
\section{Expressing $N_p$ in terms of rational functions}
\vspace{5mm}
By the definition of Bernoulli numbers we have
$$\frac{z}{e^z-1}+\frac{z}{2}=1+\sum_{n\geq 2}\frac{B_n}{n!}z^n\:,$$
where $B_n$ denotes the $n$th Bernoulli number and $|z|<2\pi$, where $z\in\mathbb{C}$.\\
Hence, if we set $z=-2\pi i f(x)/p$, where $f(x)$ is a polynomial for which $\left|f(x)\right|< p$, we obtain
$$\frac{-2\pi if(x)}{p\left(e^{-2\pi if(x)/p}-1\right)}-\frac{\pi if(x)}{p}=1+\sum_{n\geq 2}\frac{B_n}{n!}(-2\pi i)^n\frac{f(x)^n}{p^{\;n}}$$
or
$$\frac{2\pi if(x)}{p\left(1-e^{-2\pi if(x)/p}\right)}=1+\frac{\pi if(x)}{p}+\sum_{n\geq 2}G(n,x)\:,$$
where
$$G(n,x)=\frac{B_n}{n!}(-2\pi i)^n\frac{f(x)^n}{p^{\;n}}\:.$$
Therefore, we get
$$1-e^{-2\pi if(x)/p}=\frac{1}{p}2\pi if(x)\left(1+\frac{\pi if(x)}{p}+\sum_{n\geq 2}G(n,x)\right)^{-1}$$
or
$$e^{-2\pi if(x)/p}=1-2\pi if(x)\left(p+\pi if(x)+\sum_{n\geq 2}\frac{B_n}{n!}(-2\pi i)^n\frac{f(x)^n}{p^{n-1}}\right)^{-1}\:,$$
which is equivalent to
\[
e^{-2\pi if(x)/p}=1-2\pi if(x)\left(p+\pi if(x)+\sum_{n\geq 2}G_1(n,x)\right)^{-1}\:,\tag{2}
\]
where
$$G_1(n,x)=\frac{B_n}{n!}(-2\pi i)^n\frac{f(x)^n}{p^{n-1}}\:.$$
However, it is a known fact that for every integer $n\geq0$,
$$\zeta(-n)=-\frac{B_{n+1}}{n+1}\:.$$
Thus, we have
\begin{eqnarray}
	\sum_{n\geq2}G_1(n,x)&=&\sum_{n\geq1}G_1(n+1,x)\nonumber\\
	&=&\sum_{n\geq1}\frac{B_{n+1}}{(n+1)!}(-2\pi i)^{n+1}\frac{f(x)^{n+1}}{p^{n}}\nonumber\\
&=&\sum_{n\geq1}\frac{-(n+1)}{(n+1)!}\zeta(-n)(-1)^{n+1}(2\pi i)^{n+1}\frac{f(x)^{n+1}}{p^{n}}\nonumber\\
&=&\sum_{n\geq1}\frac{(-1)^n\zeta(-n)}{n!}(2\pi i)^{n+1}\frac{f(x)^{n+1}}{p^{n}}\:.\nonumber
\end{eqnarray}
But, 
$$\zeta(-2n)=0\:,$$
for every natural number $n$. Hence,
\begin{eqnarray}
	\sum_{n\geq2}G_1(n,x)&=&\sum_{n=\;\text{odd}}G_1(n,x)\nonumber\\
	&=&\sum_{n=\;\text{odd}}\frac{-\zeta(-n)}{n!}(2\pi i)^{n+1}\frac{f(x)^{n+1}}{p^{n}}\nonumber\\
	&=&\sum_{n\equiv1\;(\text{mod}\;4)}\frac{-\zeta(-n)}{n!}(2\pi i)^{n+1}\frac{f(x)^{n+1}}{p^{n}}+\sum_{n\equiv3\;(\text{mod}\;4)}\frac{-\zeta(-n)}{n!}(2\pi i)^{n+1}\frac{f(x)^{n+1}}{p^{n}}\:.\nonumber
\end{eqnarray}
Thus, we obtain
\begin{eqnarray}
\sum_{n\geq2}G_1(n,x)&=&\sum_{n\equiv1\;(\text{mod}\;4)}\frac{\zeta(-n)}{n!}(2\pi)^{n+1}\frac{f(x)^{n+1}}{p^{n}}\nonumber\\
&-&\sum_{n\equiv3\;(\text{mod}\;4)}\frac{\zeta(-n)}{n!}(2\pi)^{n+1}\frac{f(x)^{n+1}}{p^{n}}\:.\nonumber
\end{eqnarray}
Let
$$D(n,x)=\frac{\zeta(-n)}{n!}(2\pi)^{n+1}\frac{f(x)^{n+1}}{p^{n}}\:.$$
Then we can write
\[
\sum_{n\geq2}G_1(n,x)=\sum_{n\equiv1\;(\text{mod}\;4)}D(n,x)-\sum_{n\equiv3\;(\text{mod}\;4)}D(n,x)\:.\tag{3}
\]
The following functional equation holds for every integer $n$
$$\zeta(n)=2^n\pi^{n-1}\sin\left(\frac{\pi n}{2}\right)\Gamma(1-n)\zeta(1-n)\:.$$
Also, for a positive integer $n$ one has
$$\Gamma(n+1)=n!\:.$$
Thus,
\[
\zeta(-n)=-\frac{\sin\left(\frac{\pi n}{2}\right)}{2^n\pi^{n+1}}(n!)\zeta(n+1)\:.
\] 
Therefore, it is evident that
\[
D(n,x)=-\frac{2\sin\left(\frac{\pi n}{2}\right)}{p^n}\:\zeta(n+1)f(x)^{n+1}\:.\tag{4}
\]
But, for $n\equiv1\;(\bmod\;4)$ and $n\equiv3\;(\bmod\;4)$ it follows that
$$\sin\left(\frac{\pi n}{2}\right)=1\ \ \text{and}\ \ \sin\left(\frac{\pi n}{2}\right)=-1\:,$$
respectively.\\
By using (4), identity (3) implies
\begin{eqnarray}
\sum_{n\geq2}G_1(n,x)&=&\sum_{n\equiv1\;(\text{mod}\;4)}\frac{-2}{p^{\;n}}\:\zeta(n+1)f(x)^{n+1}\nonumber\\
&-&\sum_{n\equiv3\;(\text{mod}\;4)}\frac{2}{p^{\;n}}\:\zeta(n+1)f(x)^{n+1}\nonumber
\end{eqnarray}
or
$$\sum_{n\geq2}G_1(n,x)=-2S(x)\:,$$
where
\[
S(x)=\sum_{n=\;\text{odd}}\frac{\zeta(n+1)f(x)^{n+1}}{p^{\;n}}\:.\tag{5}
\]
Thus, (2) becomes
\begin{eqnarray}
e^{-2\pi if(x)/p}&=&1-2\pi if(x)\left(p+\pi if(x)-2S(x)\right)^{-1}\nonumber\\
&=&1-\frac{2\pi if(x)\left(p-2S(x)-\pi if(x)\right)}{\left(p-2S(x)\right)^2+\left(\pi f(x)\right)^2}\nonumber
\end{eqnarray}
or
\[e^{-2\pi if(x)/p}=Q(x)+iR(x)\:,\tag{6}\]
where
\[Q(x)=1-\frac{2\pi^2f(x)^2}{\left(p-2S(x)\right)^2+\left(\pi f(x)\right)^2}\ \ \text{and}\ \ R(x)=\frac{2\pi f(x)(2S(x)-p)}{\left(p-2S(x)\right)^2+\left(\pi f(x)\right)^2}\:.\tag{7}\vspace{5mm}\]
However, the interesting case is when 
$$f(x)=x^3+ax+b\:,$$
where $x,\:y\in\mathbb{M}_p$ and $a,\:b\in\mathbb{Z}$.\\
In order to use the results obtained above, we must investigate the cases for which we have $\left|f(x)\right|<p$, that is
$$-p<x^3+ax+b<p\:.$$
But, unfortunately, there are no integers $a$, $b$, for which the above inequality holds true for every prime number $p$. Hence, we shall set 
$$\tilde{f}(x)=\frac{x^3+ax+b}{p^2}$$
instead.\\
Consequently, by (6) we get
\[
e^{-2\pi i(x^3+ax+b)/p}=\left(Q(x)+iR(x)\right)^{p^2}\:.\tag{8}
\]
In this case, it is obvious that there exist integers $a$ and $b$ (for example $a=b=1$) for which the desired property is satisfied.
Therefore, by (1) and (8) we obtain the following theorem.
\begin{theorem}
For any elliptic curve $y^2\equiv x^3+ax+b\;(\bmod\;p)$, for which the integers $a$ and $b$ are such that $|\tilde{f}(x)|<p$, $\tilde{f}(x)=(x^3+ax+b)/p^2$, the number of its integral points can be expressed in the form
$$N_p=1+p+\frac{1}{p}\sum_{x,y=0}^{p-1}\sum_{m=1}^{p-1}e^{2\pi imy^2/p}\left(Q(x)+iR(x)\right)^{mp^2}\:,$$
where
\[Q(x)=1-\frac{2\pi^2\tilde{f}(x)^2}{\left(p-2S(x)\right)^2+\left(\pi \tilde{f}(x)\right)^2}\ \ \text{,}\ \ R(x)=\frac{2\pi \tilde{f}(x)(2S(x)-p)}{\left(p-2S(x)\right)^2+\left(\pi \tilde{f}(x)\right)^2}\:\]
and
$$S(x)=\sum_{n=\;\text{odd}}\frac{\zeta(n+1)\tilde{f}(x)^{n+1}}{p^{\;n}}\:.$$
\end{theorem}
\noindent However, if we express the sum involved in the representation of $N_p$ presented in the above theorem, in a slightly different form and use the same technique, we can end up with terms of power $m$ instead of $mp^2$. In addition, the following method has the advantage of dealing with elliptic curves for which we do not necessarily have 
$$\left|\frac{x^3+ax+b}{p}\right|<1\:,$$
for all $x\in\mathbb{M}_p$.\\ 
\noindent Let $f(x)=x^3+ax+b$ and
$$X=\sum_{x=0}^{p-1}e^{-2\pi imf(x)/p}\:.$$
Then, if $a,\:b\geq0$ (i.e. $f(x)$ is an increasing and positive valued function), we can express $X$ in the form
$$X=\sum_{x=0}^{L}e^{-2\pi imf(x)/p}+\sum_{x=L+1}^{p-1}e^{-2\pi imf(x)/p}\:,$$
where $L$ is such that
$$\left|\frac{f(x)}{p}\right|<1\:,$$
for every $x$, $0\leq x\leq L$. Hence, we can write
$$X=\sum_{x=0}^{L}\left(Q(x)+iR(x)\right)^{m}+\sum_{x=L+1}^{p-1}\left(e^{-2\pi if(x)/p}\right)^m\:.$$
In the second sum of the right hand side of the above relation, it always holds
$$\left|\frac{f(x)}{p}\right|\geq1\:.$$
Therefore, for the values of $x$ for which $L+1\leq x\leq p-1$, let
$$\frac{f(x)}{p}=k(x,p)+r(x,p)\:,$$
where $k(x,p)\in\mathbb{Z}$ and $r(x,p)\in\mathbb{R}$, with $0\leq r(x,p)<1$. In the following, we shall denote $k(x,p)$ and $r(x,p)$ by $k$ and $r$, respectively.\\
Then, we have
$$e^{-2\pi if(x)/p}=e^{-2\pi i(k+r)}=e^{-2\pi ir}\:.$$
But, we always have $|2\pi ir|<2\pi$. Thus, by the generating function of Bernoulli numbers 
$$\frac{z}{e^z-1}+\frac{z}{2}=1+\sum_{n\geq 2}\frac{B_n}{n!}z^n\:,$$
for $z=-2\pi i r$, we obtain by exactly the same procedure we followed previously, that
$$e^{-2\pi ir}=Q_1(r)+iR_1(r)\:,$$
where
\[Q_1(r)=1-\frac{2\pi^2r^2}{\left(1-2W(r)\right)^2+(\pi r)^2}\ ,\ R_1(r)=\frac{2\pi r\left(1-2W(r)\right)}{\left(1-2W(r)\right)^2+(\pi r)^2}\:\tag{9}\]
and 
\[W(r)=\sum_{n=\text{odd}}\zeta(n+1)r^{n+1}\:.\tag{10}\]
Hence, 
$$X=\sum_{x=0}^{L}\left(Q(x)+iR(x)\right)^{m}+\sum_{x=L+1}^{p-1}\left(Q_1(x)+iR_1(x)\right)^m$$
and therefore we obtain the following theorem.
\begin{theorem}
Let $f(x)=x^3+ax+b$, where $a$, $b\in\mathbb{Z}$. For any elliptic curve $y^2\equiv f(x)\;(\bmod\;p)$, such that $a$, $b\geq 0$, the number of its integral points can be expressed in the form
$$N_p=1+\frac{1}{p}\sum_{m=0}^{p-1}\sum_{y=0}^{p-1}e^{2\pi imy^2/p}\left(\sum_{x=0}^{L} \left(Q(x)+iR(x)\right)^{m}+\sum_{x=L+1}^{p-1}\left(Q_1(x)+iR_1(x)\right)^m\right)\:,$$
where $L$ is an integer such that $|f(x)/p|<1$, for every integer $x$, with $0\leq x\leq L$.\\
\noindent Similarly, if $a<-3(p-1)^2$ and $b\leq0$
, the number of integral points on the elliptic curve can be expressed in the form
$$N_p=1+\frac{1}{p}\sum_{m=0}^{p-1}\sum_{y=0}^{p-1}e^{2\pi imy^2/p}\left(\sum_{x=0}^{L} \left(Q_1(x)+iR_1(x)\right)^m+\sum_{x=L+1}^{p-1}\left(Q(x)+iR(x)\right)^{m}\right)\:,$$
where $L$ is an integer such that $|f(x)/p|<1$, for every integer $x$, with $L+1\leq x\leq p-1$. The functions $S(x)$, $Q(x)$, $R(x)$, $Q_1(x)$, $R_1(x)$ and $W(r)$ are defined by $(5)$, $(7)$, $(9)$ and $(10)$, respectively.
\end{theorem}
\noindent In the case where $-3(p-1)^2\leq a<0$ the function $f(x)$ is neither increasing nor decreasing for the whole interval of $x$. In that case, we would have to break the sum into more than two summands in order to apply the above technique.\\
For the cases of $a$ which we investigated, the only thing remaining is to identify the value of $L$ for a given elliptic curve $(\bmod\;p)$. But, this is always possible since it suffices to solve the equation $x^3+ax+b-p=0$ and from the solution, choose the appropriate value for $L$. \\
For example, given the elliptic curve with equation $y^2=x^3+5x+37$ and the prime number $p=1087$, it follows that $x_0=10$ is the solution of the equation $x^3+5x+37-1087=0$ and therefore we  obtain $L=9$.  
\vspace{5mm}
%
%
%
%
\section{Computing the function $S(x)$}
\vspace{5mm}
\noindent It is well known (see [20]) that
\[
\sum_{k\geq2}\frac{t^k\zeta(k)}{k}=\ln\left(\Gamma(1-t)\right)-\gamma t\:,\tag{11}
\]
for every $t\in\mathbb{R}$, such that $\left|t\right|<1$, where $\gamma$ denotes the Euler-Mascheroni constant.\\
Let $f(x)$ be a polynomial function, such that $\left|f(x)\right|<p$, where $0\leq x\leq p-1$.Then, by (11) we can write
\begin{eqnarray}
\sum_{n\geq2}\frac{f(x)^n\zeta(n)}{np^{\;n}}&=&\sum_{n=\:\text{even}}\frac{f(x)^n\zeta(n)}{np^{\;n}}+\sum_{\substack{n=\:\text{odd} \\ n\geq3}}\frac{f(x)^n\zeta(n)}{np^{\;n}}\nonumber\\
&=&\frac{1}{p}\sum_{n=\:\text{odd}}\frac{f(x)^{n+1}\zeta(n+1)}{(n+1)p^{\;n}}+\sum_{\substack{n=\:\text{odd} \\ n\geq3}}\frac{f(x)^n\zeta(n)}{np^{\;n}}\nonumber\\
&=&\ln\left(\Gamma\left(1-\frac{f(x)}{p}\right)\right)-\gamma\frac{f(x)}{p}\:.\nonumber
\end{eqnarray}
Because of the uniform convergence of the above series, differentiating we get
$$\frac{1}{p}\sum_{n=\:\text{odd}}\frac{f'(x)f(x)^{n}\zeta(n+1)}{p^{\;n}}+\sum_{\substack{n=\:\text{odd} \\ n\geq3}}\frac{f'(x)f(x)^{n-1}\zeta(n)}{p^{\;n}}=\frac{d}{dx}\left(\ln\left(\Gamma\left(1-\frac{f(x)}{p}\right)\right)-\gamma\frac{f(x)}{p}\right)\:,\footnote{By $f'(x)$ we denote the derivative $df(x)/dx$.}$$
which is equivalent to 
$$\frac{1}{p}\frac{f'(x)}{f(x)}S(x)+f'(x)\bar{S}(x)=\frac{d}{dx}\left(\ln\left(\Gamma\left(1-\frac{f(x)}{p}\right)\right)-\gamma\frac{f(x)}{p}\right)\:,$$
where
$$S(x)=\sum_{n=\:\text{odd}}\frac{f(x)^{n+1}\zeta(n+1)}{p^{\;n}}\ \ \text{and}\ \ \bar{S}(x)=\sum_{\substack{n=\:\text{odd} \\ n\geq3}}\frac{f(x)^{n-1}\zeta(n)}{p^{\;n}}\:. $$
However, for any $t>0$, it can be proved that (see [5])
$$\frac{d\ln\left(\Gamma(t)\right)}{dt}=-\gamma-\frac{1}{t}+\sum_{n\geq 1}\left(\frac{1}{n}-\frac{1}{t+n}\right)\:.$$
Hence,
$$\frac{1}{p}\frac{f'(x)}{f(x)}S(x)+f'(x)\bar{S}(x)=-\frac{f'(x)}{p}\left(-\frac{1}{1-f(x)/p}+\left(1-\frac{f(x)}{p}\right)\sum_{n\geq 1}\frac{1}{n\left(1-\frac{f(x)}{p}+n\right)}\right)\:.$$
Thus, since we are investigating the case when $f'(x)$ is either strictly positive or strictly negative (i.e. nonzero), we have
\[
S(x)+pf(x)\bar{S}(x)=f(x)\left(\frac{p}{p-f(x)}-\left(1-\frac{f(x)}{p}\right)A_1(x)\right)\:,\tag{12}
\]
where
$$A_1(x)=\sum_{n\geq 1}\frac{1}{n\left(1-\frac{f(x)}{p}+n\right)}\:.$$
But, by (11) it is evident that for every $t\in\mathbb{R}$, such that $\left|t\right|<1$, we have
$$\sum_{k\geq2}\frac{(-t)^k\zeta(k)}{k}=\ln\left(\Gamma(t+1)\right)+\gamma t\:.$$
Hence, if we substitute again $t$ by $f(x)/p$, where $\left|f(x)\right|<p$ and then calculate the derivative of both sides, we obtain
$$f'(x)\left(\sum_{n=\:\text{even}}\frac{f(x)^{n-1}\zeta(n)}{p^{\;n}}-\sum_{\substack{n=\:\text{odd} \\ n\geq3}}\frac{f(x)^{n-1}\zeta(n)}{p^{\;n}}\right)=\gamma\frac{f'(x)}{p}+\frac{d}{dx}\ln\left(\Gamma\left(\frac{f(x)}{p}+1\right)\right)$$
or
$$\frac{1}{f(x)p}S(x)-\bar{S}(x)=\frac{\gamma}{p}+\frac{1}{f'(x)}\frac{d}{dx}\ln\left(\Gamma\left(\frac{f(x)}{p}+1\right)\right)\:.$$
Equivalently,
$$S(x)-pf(x)\bar{S}(x)=\gamma f(x)+f(x)\left(-\gamma-\frac{p}{f(x)+p}+\sum_{n\geq1}\left(\frac{1}{n}-\frac{1}{\frac{f(x)}{p}+1+n}\right)\right)$$
or
\[
S(x)-pf(x)\bar{S}(x)=f(x)\left(-\frac{p}{f(x)+p}+\left(1+\frac{f(x)}{p}\right)B_1(x)\right)\:,\tag{13}
\]
where
$$B_1(x)=\sum_{n\geq 1}\frac{1}{n\left(\frac{f(x)}{p}+1+n\right)}\:.$$
Therefore, by adding (12) and (13) by parts, we obtain the following Proposition.
\begin{proposition}
Let 
$$S(x)=\sum_{n=\:\text{odd}}\frac{f(x)^{n+1}\zeta(n+1)}{p^{\;n}}\:,$$
where $f(x)$ is such that $\left|f(x)\right|<p$. Then, it holds
\[
2S(x)=f(x)\left(\frac{2pf(x)}{p^2-f(x)^2}-\left(1-\frac{f(x)}{p}\right)A_1(x)+\left(1+\frac{f(x)}{p}\right)B_1(x)\right)\:,\tag{14}
\]
where
$$A_1(x)=\sum_{n\geq 1}\frac{1}{n\left(1-\frac{f(x)}{p}+n\right)}\ \ \text{and}\ \ B_1(x)=\sum_{n\geq 1}\frac{1}{n\left(\frac{f(x)}{p}+1+n\right)}\:.$$\end{proposition}
\noindent But, it is easy to see that both series $A_1(x)$ and $B_1(x)$ converge to a real number. Therefore, (14) gives an explicit formula for the representation of $S(x)$, when $\left|f(x)\right|<p$. In order to find an upper and a lower bound for $S(x)$, we shall use the following lemma.
\begin{lemma}
Let $\alpha$, $\beta$ be real numbers, such that $\beta>\alpha>0$. Then, 
$$\sum_{k\geq1}\frac{1}{(k+\alpha)(k+\beta)}=\frac{1}{\beta-\alpha}\int_0^1\frac{x^{\alpha}-x^{\beta}}{1-x}dx\:.$$
\end{lemma}
\begin{proof}
Recall the geometric series
$$\frac{1}{1-x}=\sum_{j\geq0}x^j\:,\ \left|x\right|<1\:,$$
which converges uniformly in the interval $(-\theta,\theta)$, where $0<\theta<1$. Thus,
$$\int_0^{1-\epsilon}\frac{x^{\alpha}-x^{\beta}}{1-x}dx=\sum_{j\geq0}\int_0^{1-\epsilon}(x^{j+\alpha}-x^{j+\beta})dx=\sum_{k\geq1}\left(\frac{(1-\epsilon)^{k+\alpha}}{k+\alpha}-\frac{(1-\epsilon)^{k+\beta}}{k+\beta}\right)\:,$$
where $k=j+1$. Hence,
$$J=\lim_{\epsilon\rightarrow0^+}\int_0^{1-\epsilon}\frac{x^{\alpha}-x^{\beta}}{1-x}dx=\lim_{\epsilon\rightarrow0^+}\sum_{k\geq1}\left(\frac{(1-\epsilon)^{k+\alpha}}{k+\alpha}-\frac{(1-\epsilon)^{k+\beta}}{k+\beta}\right)$$
and thus
$$J=\lim_{\epsilon\rightarrow0^+}\sum_{k\geq1}(1-\epsilon)^k\xi(k,\alpha,\beta)\:,$$
where
$$\xi(k,\alpha,\beta)=\frac{(1-\epsilon)^{\alpha}}{k+\alpha}-\frac{(1-\epsilon)^{\beta}}{k+\beta}\:.$$
However, 
$$\xi(k,\alpha,\beta)=\frac{1-\alpha\epsilon+O(\epsilon^2)}{k+\alpha}-\frac{1-\beta\epsilon+O(\epsilon^2)}{k+\beta}\:.$$
Therefore,
$$\xi(k,\alpha,\beta)=\frac{(\beta-\alpha)(1+\epsilon k)+\epsilon^2(\delta_1k+\delta_2)}{(k+\alpha)(k+\beta)}\:,$$
where $\delta_1$ and $\delta_2$ are real constants. But, for $0<\epsilon<1$ we have
$$\frac{1}{1+\epsilon}=1-\epsilon+\epsilon^2\mp\cdots>1-\epsilon$$
and hence
$$(1-\epsilon)^k<\frac{1}{(1+\epsilon)^k}=\frac{1}{1+k\epsilon+\frac{k(k-1)}{2}\epsilon^2+\cdots}<\frac{1}{1+k\epsilon}\:.$$
Thus,
$$(1-\epsilon)^k\xi(k,\alpha,\beta)<\frac{\beta-\alpha}{(k+\alpha)(k+\beta)}+\frac{\epsilon^2(\delta_1k+\delta_2)}{(k+\alpha)(k+\beta)(1+k\epsilon)}\:.$$
Since
$$0<(1-\epsilon)^k\xi(k,\alpha,\beta)<\frac{c}{k^2}\:,$$
for $c>0$, it follows that the series $J$ converges uniformly and absolutely for every $\epsilon$, such that $0<\epsilon<1$. Therefore, we can interchange the sum with the limit for $\epsilon\rightarrow0^+$.
\end{proof}
\noindent The above lemma yields that
$$C(\lambda)=\sum_{n\geq1}\frac{1}{n(n+\lambda)}=\frac{1}{\lambda}\int_0^1\frac{u-u^{\lambda+1}}{1-u}du+\frac{1}{\lambda+1}\:,$$
for every $\lambda$, such that $\lambda>0$. Hence, for $\lambda=1\pm f(x)/p\:$, we obtain
$$\frac{3}{4}=C(2)<C(\lambda)<\zeta(2)+1=\frac{\pi^2}{6}+1$$
and evidently
$$\frac{3}{4}<A_1(x),\:B_1(x)<\frac{\pi^2}{6}+1\:,$$
where $\zeta(s)$ stands for the Riemann zeta function. Therefore, from (14) and the above relation, we obtain
$$f(x)^2\left(\frac{p}{p^2-f(x)^2}+\frac{3}{4p}\right)<S(x)<f(x)^2\left(\frac{p}{p^2-f(x)^2}+\frac{\pi^2/6+1}{p}\right)\:.$$
%
%
%
\vspace{5mm}
\section{Computing the function $W(r)$ in terms of $r$}
\vspace{5mm}
\noindent Recall that
$$W(r)=\sum_{n=\text{odd}}\zeta(n+1)r^{n+1}\:,$$
where $r=r(x,p)$. To compute $W(r)$ in terms of $r$, we shall follow a similar procedure to that of the previous section.\\
\noindent Assume that $f(x)/p\not\in\mathbb{Z}$. Then, $r\neq 0$ and also $r$ is locally differentiable. Then, if we set $t=r$ in (11), we get
$$\sum_{n\geq2}\frac{r^n\zeta(n)}{n}=\ln \Gamma(1-r)-\gamma r\:.$$
Differentiating with respect to $x$, we have
$$\frac{r'}{r}\sum_{n\geq2}r^n\zeta(n)=\frac{d}{dx}\ln \Gamma(1-r)-\gamma r'\;\;\footnote{\:By $r'$ we denote the derivative $dr(x,p)/dx$.}$$
which is equivalent to
$$\sum_{n=\text{even}}r^n\zeta(n)+\sum_{n=\text{odd}}r^n\zeta(n)=\frac{r}{r'}\frac{d}{dx}\ln \Gamma(1-r)-\gamma r\:.$$
Hence,
\[
W(r)+\overline{W}(r)=r\left(\frac{1}{1-r}+(r-1)\sum_{n\geq 1}\frac{1}{n(n+1-r)}\right)\:,\tag{15}
\]
where 
$$\overline{W}(r)=\sum_{n=\text{odd}}r^n\zeta(n)\:.$$
Similarly, by setting $t=-r$ in (11) it follows that
\[
W(r)-\overline{W}(r)=r\left(-\frac{1}{1+r}+(r+1)\sum_{n\geq 1}\frac{1}{n(n+1+r)}\right)\:,\tag{16}
\]
By (15) and (16), we obtain the following Proposition
\begin{proposition}
Let 
$$W(r)=\sum_{n=\text{odd}}\zeta(n+1)r^{n+1}\:,$$
where $r=r(x,p)$ represents the fractional part of the function $f(x)/p$. Then, it holds
\[
2W(r)=r\left(\frac{2r}{1-r^2}+(r-1)A_2(r)+(r+1)B_2(r)\right)\:,\tag{17}
\]
where 
$$A_2(r)=\sum_{n\geq 1}\frac{1}{n(n+1-r)}\ \ \text{and}\ \ B_2(r)=\sum_{n\geq 1}\frac{1}{n(n+1+r)}\:.$$
\end{proposition}
\noindent As expected, by (17) we see that the function $W(r)$ converges to a real number.\\
The above result holds true only when $f(x)/p$ is not an integer and thus $r$ is differentiable. But, if this is not the case, it is clear that $r=0$ and therefore $W(r)=0$. However, according to Lagrange's Theorem for polynomials (see [17]) we know that for any polynomial of the form
$$h(x)=a_nx^n+a_{n-1}x^{n-1}+\cdots+a_1x+a_0\:,$$
where $a_0,\:a_1,\:\ldots,\:a_n\in\mathbb{Z}$, the polynomial congruence
$$h(x)\equiv0\;(\bmod\;p)\:,$$
where $p$ is a prime number, such that $a_n\not\equiv0\;(\bmod\;p)$, has at most $n$ solutions.\\
\noindent Therefore, in our case, it is evident that the polynomial congruence
$$f(x)\equiv0\;(\bmod\;p)\:,$$
where $f(x)=x^3+ax+b$, has at most $3$ solutions. In other words, there exist at most $3$ values of $x\in\mathbb{M}_p$, such that $r=0$.
\vspace{5mm}
\section{Some remarks concerning the fractional part of $f(x)/p$}
\vspace{5mm}
\noindent In the previous section, we have computed the function $W(r)$ in terms of the fractional part $r$ of the function $f(x)/p$. Generally, we could say that the fractional part of an arbitrary function has a relatively random behavior. However, the fractional part $r$ has a more predictable behavior. In this section, we will present some basic properties of the function $r$.\\ 
\noindent It is evident that the number of factors\footnote{\;By the term \textit{factors} of $n\:!$ we refer to the integers $1,\:2,\:\ldots,\:n-1,\:n.$} of $n\:!$ which are divisible by $p$, is equal to $\left\lfloor n/p\right\rfloor$. Note that by $\left\lfloor u\right\rfloor$ we denote the greatest integer not exceeding the real number $u$. However, for any integer $n$ we know that
$$\frac{1}{p}\sum_{m=0}^{p-1}e^{2\pi imn/p}=\left\{   \begin{array}{l l}
    1\:, & \quad \text{if $n\equiv0\;(\bmod\;p)$}\vspace{2mm}\\ 
    0\:, & \quad \text{otherwise}\:.\\
  \end{array} \right.$$
Hence, it is clear that 
$$\left\lfloor \frac{n}{p}\right\rfloor=\frac{1}{p}\sum_{k=1}^{n}\sum_{m=0}^{p-1}e^{2\pi imk/p}\:.$$
Therefore, if $f(x)$ is any function of the form 
\[
f(x)=a_nx^n+a_{n-1}x^{n-1}+\cdots+a_1x+a_0\:,\tag{18}
\]
where $a_0,\:a_1,\:\ldots,\:a_n\in\mathbb{Z}$ and $x\in\mathbb{M}_p$, such that $f(x)\geq0$, it follows that 
$$\left\lfloor \frac{f(x)}{p}\right\rfloor=\frac{1}{p}\sum_{k=1}^{f(x)}\sum_{m=0}^{p-1}e^{2\pi imk/p}\:,$$
when $f(x)\geq1$, and obviously is equal to zero when $f(x)=0$. Thus, we obtain the following Proposition.
\begin{proposition}$\label{x:lem1}$
For $f(x)$ defined by $(18)$, it holds
$$\left\{\frac{f(x)}{p}\right\}=\frac{f(x)}{p}-\frac{1}{p}\sum_{k=1}^{f(x)}\sum_{m=0}^{p-1}e^{2\pi imk/p}\:,\;\footnote{\:By $\left\{u\right\}$ we denote the fractional part of the real number $u$. More specifically, we define $\left\{u\right\} =u-\left\lfloor u\right\rfloor$. Thus, $\left\{u\right\}\geq 0$ for every real number $u$.}$$
when $f(x)\geq1$.
\end{proposition}
\noindent Let
$$I(f,p)=\left\lfloor \frac{f(x)}{p}\right\rfloor\:,$$
that is
$$I(f,p)=\frac{1}{p}\sum_{k=1}^{f(x)}\sum_{m=0}^{p-1}e^{2\pi imk/p}\:.$$
Then, we have
$$I(f,p)=\frac{1}{p}\sum_{m=0}^{p-1}e^{2\pi imf(x)/p}+I(f-1,p)\:,$$
where $(f-1)(x)=f(x)-1$. Hence, we can distinguish two cases.\\
\noindent\textit{Case 1.} If $f(x)\not\equiv0\;(\bmod\;p)$, then it follows that
$$I(f,p)=I(f-1,p)\:.$$
Therefore
$$\left\{\frac{f(x)-1}{p}\right\}=\left\{\frac{f(x)}{p}\right\}-\frac{1}{p}\:.$$\\
\noindent\textit{Case 2.} If $f(x)\equiv0\;(\bmod\;p)$, then evidently $f(x)-1\not\equiv0\;(\bmod\;p)$ and thus, we obtain
$$I(f,p)=1+I(f-1,p)$$
and by Case 1, we get
$$I(f-1,p)=I(f-2,p)\:.$$
Hence
$$I(f,p)=1+I(f-2,p)\:.$$
Therefore,
$$\left\{\frac{f(x)-2}{p}\right\}=\left\{\frac{f(x)}{p}\right\}+\left(1-\frac{2}{p}\right)=1-\frac{2}{p}\:,$$
since in this case it holds $\left\{f(x)/p\right\}=0$. Hence, the following Proposition holds true.
\begin{proposition}$\label{x:lem2}$
If $f(x)\not\equiv0\;(\bmod\;p)$, then
$$\left\{\frac{f(x)}{p}\right\}=\left\{\frac{f(x)-1}{p}\right\}+\frac{1}{p}\:.$$
If $f(x)\equiv0\;(\bmod\;p)$, then
$$\left\{\frac{f(x)-2}{p}\right\}=1-\frac{2}{p}\:.$$
\end{proposition}
\noindent Since the representation of the fractional part that we have presented in Proposition $\ref{x:lem1}$ involves exponential sums of the exact same form as the ones involved in the representation of $N_p$ in (1), it is apparent that the explicit calculation of $\left\{f(x)/p\right\}$ is equally hard. For this reason, we shall limit ourselves to finding a lower bound for this fractional part. 
\begin{proposition}
For $f(x)$ defined by (18), with $f(x)\geq 1$, it holds
$$\left\{\frac{f(x)}{p}\right\}\geq\frac{f(x)}{p}-\frac{1}{p}\sum_{m=0}^{\left\lfloor p/2\right\rfloor}\min\left(\frac{p}{m},f(x) \right)-\frac{1}{p}\sum_{m=\left\lfloor p/2\right\rfloor+1}^{p-1}\min\left(\frac{1}{1-m/p},f(x) \right).$$
\end{proposition}
\begin{proof}
It is a well-known fact that for any real number $q$, it holds
$$\left|\sum_{n=B_1+1}^{B_2}e^{2\pi iqn}\right|\leq \min\left(\frac{1}{[q]}, B_2-B_1 \right)\:,$$
where $B_1$, $B_2$ are integers with $B_1<B_2$ and $[q]=\min_{\nu\in\mathbb{Z}}|q-\nu|$. Thus, we have
$$\frac{1}{p}\sum_{k=1}^{f(x)}\sum_{m=0}^{p-1}e^{2\pi imk/p}\leq\frac{1}{p}\sum_{m=0}^{p-1}\left| \sum_{k=1}^{f(x)}e^{2\pi i mk/p}\right|\leq\frac{1}{p}\sum_{m=0}^{p-1}\min\left(\frac{1}{[m/p]},f(x) \right)\:,$$
where $[m/p]=\min_{\nu\in\mathbb{Z}}\left|\frac{m}{p}-\nu \right|$.\\
\noindent Since $0\leq m/p<1$, it is evident that for the values of $m$ for which $0\leq m\leq \left\lfloor p/2 \right\rfloor$, we have $\nu=0$ and thus $[m/p]=m/p$. Similarly, for $m$ such that $\left\lfloor p/2 \right\rfloor+1\leq m\leq p-1$ we get $\nu=1$ and $[m/p]=1-m/p$. Therefore, we obtain
\begin{eqnarray}
\frac{1}{p}\sum_{k=1}^{f(x)}\sum_{m=0}^{p-1}e^{2\pi imk/p}&\leq&\frac{1}{p}\sum_{m=0}^{p-1}\min\left(\frac{1}{[m/p]},f(x) \right)\nonumber\\
&=&\frac{1}{p}\sum_{m=0}^{\left\lfloor p/2\right\rfloor}\min\left(\frac{p}{m},f(x) \right)+\frac{1}{p}\sum_{m=\left\lfloor p/2\right\rfloor+1}^{p-1}\min\left(\frac{1}{1-m/p},f(x) \right).\nonumber
\end{eqnarray}
From the above inequality and Proposition $\ref{x:lem1}$ the desired result follows.
\end{proof}
\noindent By using the above result we can calculate a lower bound for the function $W(r)$ of the previous section.
\vspace{5mm}
\section{An interesting observation}
\vspace{5mm}
\noindent In general, we have
\begin{eqnarray}
\sum_{x=0}^{p-1}e^{2\pi iax/p}=\left\{ 
  \begin{array}{l l}
    p\:, & \quad \text{if $a\equiv0\;(\bmod\;p)$}\vspace{2mm}\\ 
    0\:, & \quad \text{otherwise}\:.\\
  \end{array} \right.\nonumber
\end{eqnarray}
Hence, if we set $a=1$, it is evident that
$$\sum_{x=0}^{p-1}e^{2\pi ix/p}=0\:.$$
Moreover, the inequality $\left|x\right|<p$ holds true for every $x\in\mathbb{M}_p$. Therefore, by (6), it follows that 
$$\sum_{x=0}^{p-1}Q(x)+i\sum_{x=0}^{p-1}R(x)=0$$
and thus
$$\sum_{x=0}^{p-1}Q(x)=0\ \ \text{and}\ \ \sum_{x=0}^{p-1}R(x)=0\:,$$
where
$$Q(x)=1-\frac{2\pi^2x^2}{\left(p-2S(x)\right)^2+\left(\pi x\right)^2}\:,\vspace{2mm}$$
$$R(x)=\frac{4\pi xS(x)-2p\pi x}{\left(p-2S(x)\right)^2+\left(\pi x\right)^2}\:,$$
and
$$S(x)=\sum_{n=\text{\:odd}}\frac{x^{n+1}\zeta(n+1)}{p^{\;n}}\:.$$
Hence, it follows that for every prime number $p$ (actually the following result holds true even if $p$ is not a prime number), we have
$$p=\sum_{x=0}^{p-1}\frac{2\pi^2x^2}{\left(p-2S(x)\right)^2+\left(\pi x\right)^2}\:.$$
\\ \\
\textbf{Acknowledgments.} I would like to thank my Ph.D. advisor Professor E. Kowalski who suggested to me this area of research.
\vspace{5mm}
\begin{center}
REFERENCES
\end{center}
\vspace{5mm}
{\parindent 0mm
[1]\; T. Apostol, \textit{Introduction to Analytic Number Theory}, Springer-Verlag, New York, 1984.\\
}
{\parindent 0mm
[2]\; A. O. L. Atkin and F. Moralin, \textit{Elliptic curves and primality proving}, Math. Comp., \textbf{61}(1993), 29-68.\\
}
{\parindent 0mm
\noindent[3]\; N. D. Elkies, \textit{Elliptic and modular curves over finite fields and related computational issues}, Computational Perspectives on Number Theory: Proc. Conf. in honor of A. O. L. Atkin (D. A. Buell and J. T. Teitelbaum, eds.), AMS/International Press, 1998, pp. 21-76.\\
}
{\parindent 0mm
[4]\; S. Goldwasser and J. Kilian, \textit{Almost all primes can be quickly certified}, Proc. 18th STOC, Berkeley, May 28-30, 1986, ACM, New York, 1986, pp. 316-329.\\
}
{\parindent 0mm
\noindent[5]\; J. Havil, \textit{Gamma, Exploring Euler's Constant}, Princeton University Press, Princeton, 2003.\\
}
{\parindent 0mm
\noindent[6]\; K. Ireland and M. Rosen, \textit{A Classical Introduction to Modern Number Theory}, 2nd
Edition, GTM 84, Springer-Verlag, New York, 1990.\\
}
{\parindent 0mm
\noindent[7]\; A. Ivi\'{c}, \textit{The Riemann Zeta-Function: The Theory of the Riemann Zeta-Function with Applications}, John Wiley \& Sons, Inc., New York, 1985.\\
}
{\parindent 0mm
[8]\; H. Iwaniec and E. Kowalski, \textit{Analytic Number Theory}, A.M.S Colloq. Publ. 53,
A.M.S, 2004.\\
}
{\parindent 0mm
[9]\; N. Koblitz, \textit{Elliptic curve cryptosystems}, Math. Comp., \textbf{48}(1987), 203-209.\\
}
{\parindent 0mm
[10]\; N. Koblitz, \textit{A Course in Number Theory and Cryptography}, Springer-Verlag, New York, 1994.\\
}
{\parindent 0mm
\noindent[11]\; N. Korobov, \textit{Exponential Sums and their Applications}, Kluwer Academic Publishers, Dordrecht, Boston, London, 1992.\\
}
{\parindent 0mm
\noindent[12]\; E. Kowalski, \textit{Exponential sums over finite fields, I: Elementary methods},\\ http://www.math.ethz.ch/~kowalski/exp-sums.pdf.\\
}
{\parindent 0mm
[13]\; H. W. Lenstra, \textit{Factoring integers with elliptic curves}, Annals Math., \textbf{126}(3)(1987), 649-673.\\
}
{\parindent 0mm
\noindent [14]\; V. Miller, \textit{Uses of elliptic curves in cryptography}, Advances in Cryptology: Proc. of Crypto '85, Lecture Notes in Computer Science, \textbf{218}(1986), Springer-Verlag, New York, pp. 417-426.\\
}
{\parindent 0mm\noindent[15]\; S. J. Miller and R. Takloo-Bighash, \textit{An Invitation to Modern Number Theory},
Princeton University Press, Princeton and Oxford, 2006.\\
}
{\parindent 0mm
\noindent [16]\; P. L. Montgomery, \textit{Speeding the Pollard and elliptic curve methods for factorizations}, Math. Comp., \textbf{48}(1987), 243-264.\\
}
{\parindent 0mm
\noindent[17]\; M. Th. Rassias, \textit{Problem-Solving and Selected Topics in Number Theory: In the Spirit of the Mathematical Olympiads}, Springer, New York, 2011.\\
}
{\parindent 0mm
\noindent [18]\; R. Schoof, \textit{Elliptic curves over finite fields and the computation of square roots mod p}, Math. Comp., \textbf{44}(170)(1985), 483-494.\\
}
{\parindent 0mm
\noindent[19]\; J. Silverman, \textit{The Arithmetic of Elliptic Curves}, Graduate Texts in Mathematics
\textbf{106}, Springer-Verlag, New York, 1986.\\
}
{\parindent 0mm
\noindent[20]\; H. M. Srivastava and J. Choi, \textit{Series Associated with the Zeta and Related Functions}, Kluwer Academic Publishers, Dordrecht, Boston, London, 2001.\\
}
{\parindent 0mm
\noindent[21]\; L. C. Washington, \textit{Elliptic Curves-Number Theory and Cryptography}, CRC Press, London, New York, 2008.\\
}
\end{document}